\documentclass{article}

\usepackage{amsmath, amsthm, amsfonts, amssymb, faktor, mathrsfs, textcomp, epsfig, graphicx}
\usepackage[curve, arrow]{xypic}
\usepackage{mathrsfs}
\usepackage{url}
\usepackage[utf8]{inputenc}

\newtheorem{theorem}{Theorem}[section]
\newtheorem{proposition}[theorem]{Proposition}
\newtheorem{lemma}[theorem]{Lemma}
\newtheorem{corollary}[theorem]{Corollary}

\theoremstyle{definition}
\newtheorem{definition}[theorem]{Definition}
\newtheorem{proposition-definition}[theorem]{Proposition-Definition}
\newtheorem{example}[theorem]{Example}

\theoremstyle{remark}
\newtheorem{remark}[theorem]{Remark}

\numberwithin{equation}{section}

\newcommand{\op}{\operatorname}
\newcommand{\bull}{\mbox{\tiny{$\bullet$}}}

\begin{document}
\title{Cohomological invariants of algebraic stacks}

\author{Roberto Pirisi}
%\email{rpirisi@uottawa.ca}
%\address{Department of Mathematics and Statistics, University of Ottawa, Kind Edward St., K1N6N5, Ottawa, Canada}
%\subjclass[2010]{14D23, 14F20}
%\keywords{algebraic stack, cohomological invariant}

\date{March 14, 2016}

\maketitle

\begin{abstract}
The purpose of this paper is to lay the foundations of a theory of invariants in étale cohomology for smooth Artin stacks. We compute the invariants in the case of the stack of elliptic curves, and we use the theory we developed to get some results regarding Brauer groups of algebraic spaces.
\end{abstract}
\section{Introduction}\label{sec:intro}

\emph{Some notation: we fix a base field $k_0$ and a positive number $p$. We will always assume that the characteristic of $k_0$ does not divide $p$. If $X$ is a $k_0$-scheme we will denote by $\op{H}^{i}(X)$ the étale cohomology ring of $X$ with coefficients in $\mu_p^{\otimes i}$ (here $\mu_p^{\otimes 0}:=\mathbb{Z}/p\mathbb{Z}$), and by $\op{H}^{\bull}(X)$ the direct sum $\oplus_i \op{H}^i(X)$. If $R$ is a $k_0$-algebra, we set $\op{H}^{\bull}(R)=\op{H}^{\bull}(\op{Spec}(R))$.}
 
An early example of cohomological invariants dates back to Witt's seminal paper \cite{Wi37}, where the Hasse-Witt invariants of quadratic forms were defined. Many other invariants of quadratic forms, such as the Stiefel-Whitney classes and the Arason invariant were studied before the general notion of étale cohomological invariant was introduced.

This was inspired by the theory of characteristic classes in topology, and is naturally stated in functorial terms as follows.

Denote by (Field/$k_0$) the category of extensions of $k_0$. Its objects are field extensions of $k_0$, and the arrows are morphisms of $k_0$-algebras. We think of $\op{H}^{\bull}$ as a functor from (Field/$k_0$) to the category of graded-commutative $\mathbb{Z}/p$-algebras.

Assume that we are given a functor $F:\left( \text{Field} / k_0 \right) \rightarrow \left( \text{Set} \right)$. A cohomological invariant of $F$ is a natural transformation $F \rightarrow \op{H}^{\bull}$. The cohomological invariants of $F$ form a graded-commutative ring $\op{Inv}^{\bull}(F)$. 

Given an algebraic group $G$, one can define the cohomological invariants of $G$ as $\op{Inv}^{\bull}(\textit{Tors}_G)$, where $\textit{Tors}_G$ is the functor sending each extension $K$ of $k_0$ to the set of isomorphism classes of $G$-torsors over $\op{Spec}(K)$. The book \cite{GMS03} is dedicated to the study of cohomological invariants of algebraic groups; since many algebraic structures correspond to $G$-torsors for various groups $G$ (some of the best known examples are étale algebras of degree $n$ corresponding to $S_n$-torsors, nondegenerate quadratic forms of rank $n$ corresponding to $O_n$-torsors, and central simple algebras of degree $n$ corresponding to $PGL_n$-torsors) this gives a unified approach to the cohomological invariants for various types of structures.

Suppose that $\mathscr{M}$ is an algebraic stack smooth over $k_0$, for example the stack $\mathscr{M}_g$ of smooth curves of genus $g$ for $g \geq 2$. We can define a functor $F_{\mathscr{M}}: \left( \text{Field}/k_0 \right) \rightarrow \left( \text{Set} \right) $  by sending a field $K$ to the set of isomorphism classes in $\mathscr{M}(K)$, and thus a ring $\op{GInv}^{\bull}(\mathscr{M})$ of \emph{general} cohomological invariants of $\mathscr{M}$ defined as natural transformations from $F_{\mathscr{M}}$ to $\op{H}^{\bull}$.

The definition above recovers the definition of cohomological invariants of an algebraic group $G$ when $\mathscr{M}=BG$, the stack of $G$-torsors. However, when the objects of $\mathscr{M}$ are not étale locally isomorphic as in the case of $BG$, this is not the right notion. For example, if $\mathscr{M}$ has a moduli space $M$, every object of $\mathscr{M}(K)$ determines a point $p \in M$, corresponding to the composite $\op{Spec}(K) \rightarrow \mathscr{M} \rightarrow M$. If we denote by $F^p_{\mathscr{M}}$ the subfunctor of $F_{\mathscr{M}}$ corresponding to isomorphism classes of objects in $\mathscr{M}$ with image $p$ it is easy to see that

\[ \op{GInv}^{\bull}\left( F_{\mathscr{M}} \right)=\prod_{p\in M} \op{GInv}^{\bull}\left( F_{\mathscr{M}}^p \right)\]
 
This is clearly too large to be interesting. We need to impose a continuity condition to be able to compare cohomological invariants at points of $\mathscr{M}$ with different images in $M$.

The following condition turns out to be the correct one when $\mathscr{M}$ is smooth. Let $R$ be an Henselian $k_0$-algebra that is a discrete valuation ring, with fraction field $K$ and residue field $k$. We have induced cohomology maps $\op{H}^{\bull}(R) \rightarrow \op{H}^{\bull}(k)$ and $\op{H}^{\bull}(R) \rightarrow \op{H}^{\bull}(K)$; the first is well-known to be an isomorphism \cite[09ZI]{St15}. By composing the second map with the inverse of the first we obtain a ring homomorphism $j_R: \op{H}^{\bull}(k) \rightarrow  \op{H}^{\bull}(K)$. Furthermore, from an object $\xi \in \mathscr{M}(R)$ we obtain objects $\xi_k \in \mathscr{M}(k)$ and $\xi_K \in \mathscr{M}(K)$.

\begin{definition}
A general cohomological invariant $\alpha \in \op{GInv}(\mathscr{M})$ is \emph{continuous} if for every Henselian DVR as above, and every $\xi \in \mathscr{M}(R)$ we have
\[ j_R(\alpha(\xi_k)) = \alpha(K).\]
\end{definition}

Continuous cohomological invariants form a graded subring $\op{Inv}^{\bull}(\mathscr{M})$ of $\op{GInv}(\mathscr{M})$.

When $\mathscr{M}$ is equal to $BG$ all (general) cohomological invariants are continuous (this is a result by Rost \cite[11.1]{GMS03}). Thanks to this from now on we will be able to abuse notation and forget about the word continuous without causing confusion.

 If $\mathscr{M}=X$, where $X$ is a scheme, the cohomological invariants of $X$ can be described as the Zariski sheafification of the presheaf $U \rightarrow \op{H}^{\bull}(U)$, which is in turn equal to the unramified cohomology of $X$ (see \cite[4.2.2]{BO74}).
 
 The Zariski topology is clearly too coarse for a similar result to hold for algebraic stacks. The appropriate class of morphisms to study turns out to be the following.
 
 \begin{definition}
 Let $\mathscr{M},\mathscr{N}$ be algebraic stacks. A \emph{smooth-Nisnevich} covering $f:\mathscr{M} \rightarrow \mathscr{N}$ is a smooth representable morphism such that for every field $K$ and every object $\xi \in \mathscr{N}(K)$ we have a lifting 
 \[ \op{Spec}(K) \xrightarrow{\xi'} \mathscr{M} \xrightarrow{f} \mathscr{N}, \quad f \circ \xi' \simeq \xi.\]
 
 \end{definition}

Cohomological invariants satisfy the sheaf conditions with respect to smooth-Nisnevich morphisms, see (\ref{sheaf}). Namely, given a smooth-Nisnevich morphism $f: \mathscr{M} \rightarrow \mathscr{N}$ the cohomological invariants $\op{Inv}^{\bull}(\mathscr{N})$ are the equalizer of the following diagram

\begin{center}
$\xymatrixcolsep{5pc}
\xymatrix{\op{Inv}^{\bull}(\mathscr{N}) \ar@{->}[r]^{f^{*}} & \op{Inv}^{\bull}(\mathscr{M}) \ar@/^/[r]^{\op{Pr_1}^{*}} \ar@/_/[r]_{\op{Pr_2}^{*}} & \op{Inv}^{\bull}(\mathscr{M} \times_{\mathscr{N}} \mathscr{M}) }
$ 
 
 \end{center}
 
Using this and the result for schemes we obtain the following (thm. \ref{princ}):

\begin{theorem}\label{teo1}
Let $\mathscr{M}$ be an algebraic stack smooth over $k_0$. Denote $\mathscr{C}$ to be the site of $\mathscr{M}$-schemes with covers given by Nisnevich morphisms. Consider the sheafification $\mathscr{H}^{\bull}$ of the presheaf $U \rightarrow \op{H}^{\bull}(U)$ in $\mathscr{C}$. Then
\[ \op{Inv}^{\bull}(\mathscr{M})=H^0(\mathscr{M},\mathscr{H}^{\bull}).\]
\end{theorem}
 
From this we obtain the fundamental corollary

\begin{corollary} 
$ $
\begin{enumerate}
\item Let $\mathscr{E} \rightarrow \mathscr{M}$ be a vector bundle. Then the pullback $\op{Inv}^{\bull}(\mathscr{M}) \rightarrow \op{Inv}^{\bull}(\mathscr{E})$ is an isomorphism.
\item Let $\mathscr{N}$ be a closed substack of codimension $2$ or more. Then the pullback $\op{Inv}^{\bull}(\mathscr{M}) \rightarrow \op{Inv}^{\bull}(\mathscr{M}\smallsetminus \mathscr{N})$ is an isomorphism.
\end{enumerate}
\end{corollary} 

If $\mathscr{M}$ is a smooth quotient stack, that is, a quotient $\left[ X / G \right]$, where $G$ is an affine algebraic group over $k_0$ acting on a smooth algebraic space $X$, then there exists a vector bundle $\mathscr{E} \rightarrow \left[ X / G \right]$ with an open substack $V\rightarrow \mathscr{E}$, where $V$ is an algebraic space, such that the complement of $V$ in $\mathscr{E}$ has codimension at least $2$ \cite[7,Appendix]{EG96}. Then $\op{Inv}^{\bull}(\left[ X / G \right])=\op{Inv}^{\bull}(V)$ (in the case of $\mathscr{M}=BG$ this result was obtained by Totaro in \cite[Appendix C]{GMS03}). For many purposes, this allows to reduce the study of cohomological invariants of general smooth algebraic stacks to the case of smooth algebraic spaces.

When $k_0$ contains a primitive root of unit we can also use the results above to get a spectral sequence involving the cohomological invariants of a smooth quotient stack, following Guillot \cite[sec.5]{Gu08}, which in turn is adapting the Bloch-Ogus spectral sequence and its reinterpretation by Rost. The sequence reads $$E_2^{r,q} = A^r_G(X, \op{H}^{q-r}) \Rightarrow H^{r+q}_{\textit{\'et}}(\left[X/G \right],\mathbb{Z}/p\mathbb{Z})$$
where $A^r_G(X, \op{H}^{q-r})$ is the $r$-codimensional component of the equivariant Chow ring with coefficients in $\op{H}^{\bull}$, and $A^0_G(X, \op{H}^{q})=\op{Inv}^{q}(\left[X/G \right])$.

The main theorem fails as soon as $\mathscr{M}$ is no longer normal as the continuity condition fails to provide sufficient constraints on the invariants. It is possible that some of the theory may work for some intermediate condition between the two.

There is a natural way to extend the notion of cohomological invariants to allow for coefficients in any functor that satisfies the property of Rost's cycle modules \cite{Ro96}. The main example for these, besides étale cohomology, is Milnor's $K$-theory.

 The price to pay for the generalization is a slightly less clear characterization of cohomological invariants. As the isomorphism between the étale cohomology of a Henselian DVR $(R,v)$ and that of its residue field is no longer true, we do not have any natural map $M(k(v))\rightarrow M(k(R))$, and instead we use a different group which Rost denotes $M(v)$. There is a canonical projection $p:M(k(R)) \rightarrow M(v)$, and a natural inclusion $i: M(k(v))\rightarrow M(v)$, so we can reformulate the continuity condition in the following way
 
 \begin{definition}
 Let $\mathscr{X}$ be an algebraic stack, and $M$ a cycle module. A cohomological invariant for $\mathscr{X}$ with coefficients in $M$ is a natural transformation from the functor of points
 
 \[ F_{\mathscr{X}}:(\mbox{field}/k_0) \rightarrow (\mbox{set}) \] 
 to $M$, satisfying the following property: given a Henselian DVR $R$ and a map $\op{Spec}(R) \rightarrow \mathscr{X}$ we have $p(\alpha(k(R))=i(\alpha(k(v))$ in $M(v)$.
 \end{definition}
 
When $M=\op{H}^{\bull}$, the definition is equivalent to the original one. The main theorem still holds, using Rost's $0$-codimensional Chow group with coefficients in $M$ in place of the sheaf $\mathscr{H}$ (thm. \ref{A0}).

 \begin{theorem}\label{teo2}
Let $\mathscr{X}$ be an algebraic stack smooth over $k_0$, and let $M$ be a cycle module. Denote $\mathscr{C}$ to be the site of $\mathscr{M}$-schemes with covers given by Nisnevich morphisms. Consider the sheafification $\mathscr{H}^{\bull}$ of the presheaf $U \rightarrow \op{H}^{\bull}(U)$ in $\mathscr{C}$. Then
\[ \op{Inv}^{\bull}(\mathscr{X},M)=H^0(\mathscr{X},A^{0}(-,M)).\]
\end{theorem}

We can use the new definition to infer some properties of the cohomological Brauer groups of smooth quasi-separated algebraic spaces in characteristic zero, see (\ref{brauer}). Recall that the cohomological Brauer group $\op{Br}'(X)$ is defined as the torsion in $H^2_{\textit{\'et}}(X,G_m)$, and the classical Brauer group $\op{Br}(X)$ injects to it. To the author's best knowledge, the following result is not present in the literature.

\begin{theorem}
Suppose that $\op{char}(k_0)=0$. Then:
\begin{enumerate}
\item Let $X$ be a quasi-separated algebraic space, smooth over $K_0$. Then the functor $U \rightarrow \op{Br}'(U)$ is a sheaf in the Nisnevich site of $X$.
\item (Purity) Let $X$ be a quasi-separated algebraic space, smooth over $K_0$. If $Z$ is a closed subspace of codimension $2$ or greater, then $\op{Br}'(X)= \op{Br}'(X\smallsetminus Z)$.
\item (Birational invariance) Let $X,Y$ be quasi separated algebraic spaces smooth and proper over $k_0$. If they are birational to each other then $\op{Br}'(X)=\op{Br}'(Y)$.
\end{enumerate}
\end{theorem}

\emph{Aknowledgements:} This paper is part of the author's Ph.D. thesis at the Scuola Normale Superiore of Pisa. I thank Angelo Vistoli for being the best Ph.D. advisor I could hope for. I thank Burt Totaro, Zinovy Reichstein and Arthur Ogus for their many helpful suggestions and careful reading of my material. 
I thank Benjamin Antineau, John Calabrese, Raju Krishnamoorty, Mattia Talpo, Fabio Tonini, Sophie Marques and Giulio Codogni for many useful discussions.
\section{Cohomological invariants}

We begin by defining the notion of a cohomological invariant and exploring some of its consequences.

\begin{lemma}\label{coomhens}
Let $R$ be an Henselian ring, $k$ its residue field. The closed immersion $\op{Spec}(k) \rightarrow \op{Spec}(R)$ induces an isomorphism of graded rings $\op{H}^{\bull}(R) \rightarrow \op{H}^{\bull}(k)$.
\end{lemma}
\begin{proof}
This is Gabber's theorem, \cite[09ZI]{St15}.
\end{proof}

Given a Henselian ring $R$, with residue field $k$ and field of fractions $K$, we can construct a map 

 \[ j: \op{H}^{\bull}(\op{Spec}(k)) \rightarrow \op{H}^{\bull}(\op{Spec}(K))\] 
 by composing the inverse for the isomorphism $\op{H}^{\bull}(\op{Spec}(R)) \rightarrow \op{H}^{\bull}(\op{Spec}(k))$ and the pullback $\op{H}^{\bull}(\op{Spec}(R)) \rightarrow \op{H}^{\bull}(\op{Spec}(K))$. The same map $j$ is obtained by different methods in \cite[7.6,7.7]{GMS03}.
 
\begin{definition}
Let $\mathscr{M}$ be an algebraic stack. A \emph{cohomological invariant} of $\mathscr{M}$ is a natural transformation
\begin{center}
$\alpha: \faktor{\op{Hom}(-,\mathscr{M})}{\cong} \rightarrow \op{H}^{\bull}(-)$
\end{center}
seen as functors from $\left(\faktor{\textit{field}}{k_0}\right)$ to $(\textit{set})$, satisfying the following property:
\item Let $X$ be the spectrum of a Henselian DVR, $p,P$ its closed and generic points. Then given a map $f: X \rightarrow \mathscr{M}$, we have 
\begin{equation}\label{def1}
\alpha(f\circ P) = j(\alpha(f\circ p)).
\end{equation}
The grading and operations on cohomology endow the set of cohomological invariants of $\mathscr{M}$ with the structure of a graded ring, which we will denote $\op{Inv}^{\bull}(\mathscr{M})$.
\end{definition}

As stated in the introduction, if $\mathscr{M}$ is the stack of $G$ torsors for an algebraic group $G$ condition \ref{def1} is automatic. This is proven in \cite[11.1]{GMS03}. The continuity condition has the very important consequence of tying cohomological invariants to another well-known invariant, unramified cohomology.

\begin{definition}
Let $K$ be a field. Given a discrete valuation $v$ on $K$ there is a standard residue map $\partial_{v} : \op{H}^{\bull}(K) \rightarrow \op{H}^{\bull}(k(v))$ of degree $-1$ (for example it is constructed in \cite[sec.6-7]{GMS03}).

Consider a field extension $K/k$. The \emph{unramified cohomology} $\op{H}^{\bull}_{nr}(K/k)$ is the kernel of the map
 
 \[\partial: \op{H}^{\bull}(K) \rightarrow \bigoplus_{v} \op{H}^{\bull}(k(v)) \]
 defined as the direct sum of $\delta_{v}$ over all discrete valuations on $K$ that are trivial on $k$. Given a normal scheme $X$ over $k$ we define the unramified cohomology $\op{H}_{nr}^{\bull}(X)$ to be the kernel of the map
 
 \[\partial_X: \op{H}^{\bull}(K) \rightarrow \bigoplus_{\substack{p \in X^{(1)}}} \op{H}^{\bull}(k(v_p)) \]
here the sum is over all valuations induced by a point of codimension one. For an irreducible scheme $X$ we define $\op{H}_{nr}^{\bull}(X)$ to be the unramified cohomology of its normalization.  If $R$ is a $k_0$-algebra we denote $\op{H}_{nr}^{\bull}(R)= \op{H}_{nr}^{\bull}(\op{Spec}(R))$.
\end{definition}

When $X$ is an irreducible and normal scheme, proper over $\op{Spec}(k)$, we have the equality $\op{H}^{\bull}_{nr}(X)=\op{H}^{\bull}_{nr}(k(X)/k)$, making unramified cohomology a birational invariant. It was first introduced to study rationality problems in the guise of the unramified Brauer group \cite{Sa84}, and then generalized to all degrees \cite{CO89}.

\begin{remark}\label{unramif}
Let $X$ be the spectrum of a DVR $R$, with residue field $k$ and fraction field $K$. The residue map $\partial_v: \op{H}^{\bull}(K) \rightarrow \op{H}^{\bull}(k)$ has the property that if we consider the map $X^h \xrightarrow{h} X$ induced by the Henselization of $R$, and the residue map $\partial_{v'}$ on $X^h$ we have (see \cite[sec 1, R3a]{Ro96})
\[\partial_{v'} \circ h^* = \partial_v : \op{H}^{\bull}(K) \rightarrow \op{H}^{\bull}(k).\]

As we have $\partial_{v'} \circ j = 0$ \cite[7.7]{GMS03}, we conclude that given a normal scheme $X$ and a cohomological invariant $\alpha \in \op{Inv}^{\bull}(X)$ the value of $\alpha$ at the generic point $\xi_X$ belongs to $\op{H}^{\bull}_{nr}(X)$.
\end{remark}

In fact we can say much more. There is an obvious pullback on cohomological invariants:

\begin{definition}
Given a morphism $f: \mathscr{M} \rightarrow \mathscr{N}$, we define the pullback morphism $f^* : \op{Inv}^{\bull}(\mathscr{N}) \rightarrow \op{Inv}^{\bull}(\mathscr{M})$ by setting $f(\alpha)(p) = \alpha(p\circ f)$.
\end{definition}

Given any map $f: X \rightarrow \mathscr{M}$ from an irreducible scheme $X$, with generic point $\xi_X$, consider the pullback of a cohomological invariant $\alpha \in \op{Inv}(\mathscr{M})$. We can apply the remark above to $f^*(\alpha)$, obtaining $\alpha(\xi_X)=f^*(\alpha)(\xi_X) \in \op{H}^{\bull}_{nr}(X)$. In the case of a closed immersion $V \rightarrow X$, with $V$ irreducible, this says that the value $\alpha(\xi_V)$ belongs to the unramified cohomology $\op{H}^{\bull}_{nr}(V)$.

It is immediate to check that the cohomological invariants of the spectrum of a field are canonically isomorphic to its étale cohomology. Moreover, there is a natural map from étale cohomology to cohomological invariants sending an element $x \in \op{H}^{\bull}(\mathscr{M})$ to the invariant $\tilde{x}$ defined by $\tilde{x}(p)=p^*(x)$. The étale cohomology of an algebraic stack is defined as the sheaf cohomology in its Lisse-étale site \cite[01FQ, 0786]{St15}.

 If $R,k,K$ are as in the definition the elements $\tilde{x}(p), \tilde{x}(P)$ are both pullbacks of $f^*(x) \in \op{H}^{\bull}(R)$, and the functoriality of pullback allows us to conclude that the continuity condition (\ref{def1}) is fulfilled. This map is clearly not injective, as the next example shows.

\begin{example}\label{Ex1}
An easy example of the map $\op{H}^{\bull}(\mathscr{M}) \rightarrow \op{Inv}^{\bull}(\mathscr{M})$ not being injective comes from computing the cohomology ring of $\op{B}(\mathbb{Z}/2)$ with coefficients in $\mathbb{F}_2$ over an algebraically closed field of characteristic different from $2$. By the Hochschild-Serre spectral sequence and group cohomology we obtain $\op{H}^{\bull}(\op{B}(\mathbb{Z}/2)) = \mathbb{F}_{2}[t]$, where $\op{deg}( t )= 1$, while $\op{Inv}(\op{B}(\mathbb{Z}/2),p) = \faktor{\mathbb{F}_{2}[t]}{t^2}$ when $p=2$ \cite[16.2]{GMS03}.
\end{example}

\section{The smooth-Nisnevich sites}

We want to make $\op{Inv}^{\bull}$ into a sheaf for an appropriate Grothendieck topology. It cannot be a sheaf in the étale topology as the pullback through a finite separable extension is not in general injective for étale cohomology. The Zariski topology is not satisfactory as algebraic stacks do not have Zariski covers by schemes. The Nisnevich topology, consisting of étale morphisms $X \rightarrow Y$ having the property that any map from the spectrum of a field to $Y$ lifts to $X$ looks like a promising compromise, at least for Deligne-Mumford stacks. Unfortunately it still does not fit our needs, as the following example shows:

\begin{example}
There are Deligne-Mumford stacks that do not admit a Nisnevich covering by a scheme, and this is a very common occurrence.

Consider $\mathscr{M}=B(\mu_2)$. The $\mu_2$-torsor $P=\op{Spec}(k(t))\rightarrow \op{Spec}(k(t^2))$ is not obtainable as the pullback of any torsor $T \rightarrow \op{Spec}(k)$ with $k$ finite over $k_0$. This shows that given an étale map $X \rightarrow \mathscr{M}$ there cannot be a lifting of $P$ to $X$, as any point of $X$ will map to a torsor $T$ as above.

This is due to the essential dimension (\cite{BRA11},\cite{BR97},\cite{BF03}): whenever we have a strict inequality $\op{dim}(\mathscr{M}) < \op{ed}(\mathscr{M})$ there can be no Nisnevich cover of $\mathscr{M}$ by a scheme. For an irreducible Deligne-Mumford stack this happens whenever the generic stabiliser is not trivial. The stack $\mathscr{M}=B(\mu_2)$ we considered for this example has dimension equal to zero and essential dimension equal to one.
\end{example}

To solve this problem, we admit all smooth maps satisfying the lifting property:

\begin{definition}
Let $f: \mathscr{M} \rightarrow \mathscr{N}$ be a representable morphism of algebraic stacks, $p \in \mathscr{N}(K)$. Then $f$ is a \emph{Smooth-Nisnevich} neighbourhood of $p$ if it is smooth and there is a representative $\op{Spec}(K) \rightarrow \mathscr{N}$ of the isomorphism class of $p$ such that we have a lifting 
\begin{center}

$\xymatrixcolsep{5pc}
\xymatrix{& \mathscr{M} \ar@{->}[d]^{f} \ar@{<-}[dl]_{p'}\\ 
\op{Spec}(k) \ar@{->}[r]^{p} & \mathscr{N}}$

\end{center}

$f$ is a \emph{Smooth-Nisnevich} cover if for every field $K$ and every $p \in \mathscr{N}(K)$ it is a \emph{Smooth-Nisnevich} neighbourhood of $p$.

Note that if $f$ is a \emph{Smooth-Nisnevich neighbourhood} of $p$, then given any representative of $p$ such lifting exists.

\end{definition}

This topology looks awfully large. Luckily, as we will prove shortly, when restricted to schemes it coincides with the usual Nisnevich topology. Recall that given a quasi-separated algebraic space we always have an étale Nisnevich cover by a scheme \cite[6.3]{Kn71}, so we can trivially extend the Nisnevich topology to the category of quasi-separated algebraic spaces.

\begin{proposition}\label{spezz}
Denote by $\faktor{\op{Alg}}{k_0}$ the category of quasi-separated algebraic spaces over the spectrum of $k_0$. Let $F$ be a presheaf on $\faktor{\op{Alg}}{k_0}$. The sheafification of $F$ with respect to the Ninsnevich topology is the same as its sheafification with respect to the Smooth-Nisnevich topology.
\end{proposition}
\begin{proof}
We need to show that any smooth-Nisnevich cover has a section Nisnevich locally. As we can take a Nisnevich cover of an algebraic space that is a scheme, we can restrict to schemes. Recall (see \cite[ch.6,2.13-2.14]{Li02}) that if $f: X \rightarrow \op{Spec}(R)$ is a smooth morphism from a scheme to the spectrum of an Henselian ring with residue field $k$, given a $k$-rational point $p$ of $X$ there is always a section of $f$ sending the closed point of $\op{Spec}(R)$ to $p$. Let now $X \xrightarrow{\pi} Y$ be a smooth-Nisnevich cover. Consider the diagram:

\begin{center}
$\xymatrixcolsep{5pc}
\xymatrix{ X\times_{Y}\op{Spec}(\mathscr{O}^{h}_{Y,p}) \ar@{->}[r]^{\op{pr}_1} \ar@{->}[d]^{\op{pr}_2} & X \ar@{->}[d]^{\pi}\\ 
\op{Spec}(\mathscr{O}^{h}_{Y,p}) \ar@{->}[r]^{j} & Y}$
\end{center}

As $p$ lifts to a point of $X$, the left arrow has a section. The scheme $\op{Spec}(\mathscr{O}^{h}_{Y,p})$ is the direct limit of all the Nisnevich neighbourhoods of $p$, so there is a Nisnevich neighbourhood $U_p$ of $p$ with a lifting to $X$. By taking the disjoint union over the points of $X$ we obtain the desired Nisnevich local section.
\end{proof}

In particular, the local ring at a point of a scheme in the smooth-Nisnevich topology is still the Henselization of the local ring in the Zariski topology, and in general if we consider a category of algebraic spaces containing all étale maps the topoi induced by the Nisnevich and smooth-Nisnevich topology will be equivalent.

The smooth-Nisnevich topology has the annoying problem that an open subset of a vector bundle may not be a covering when working over finite fields. To solve this, see (\ref{openvec}) below, we introduce the following larger topologies.

\begin{definition}
Let $f: \mathscr{M} \rightarrow \mathscr{N}$ be a representable morphism of algebraic stacks, $p \in \mathscr{N}(K)$. Let $m$ be a non negative integer. Then $f$ is a \emph{m-Nisnevich} (resp. \emph{smooth m-Nisnevich}) neighbourhood of $p$ if it is étale (resp. smooth) and there are finite separable extensions $K_1, \ldots, K_r$ of $K$ with liftings
\begin{center}

$\xymatrixcolsep{5pc}
\xymatrix{\op{Spec}(K_i) \ar@{->}[d]^{\phi_i} \ar@{->}[r]_{p'}& \mathscr{M} \ar@{->}[d]^{f} \\ 
\op{Spec}(K) \ar@{->}[r]^{p} & \mathscr{N}}$

\end{center}

Where $(\left[ K_1: K \right], \ldots , \left[ K_r:K \right], m)=1$.

$f$ is an \emph{m-Nisnevich} (resp. \emph{smooth m-Nisnevich}) cover if for every field $K$ and every $p \in \mathscr{N}(K)$ it is an \emph{m-Nisnevich} (resp. \emph{smooth m-Nisnevich}) neighbourhood of $p$.
\end{definition}

The $m$-Nisnevich topology strictly contains the Nisnevich topology for all $m$. If $m=1$, we get the étale topology, and the $m$-Nisnevich topology contains the $n$-Nisnevich topology if and only if the prime factors of $n$ divide $m$.

Proposition (\ref{spezz}) holds \emph{verbatim} for the $m$-Nisnevich and smooth $m$-Nisnevich topologies, as we can just repeat the argument adding some base changes.

\begin{lemma}\label{openvec}
Let $V \rightarrow \op{Spec}(k)$ be a vector space, and $U \rightarrow V$ a non-empty open subset. Then $U$ is a smooth $m$-Nisnevich cover of $\op{Spec}(k)$ for all $m$, and if $k$ is infinite it is a smooth-Nisnevich cover.
\end{lemma}
\begin{proof}
It suffices to prove this this for a $V=\mathbb{A}^1_k$. The statement for $k$ infinite is obvious. Suppose now that $k$ is finite. A closed subset $Z \subsetneq V$ only contains a finite number of closed point, so for any prime $q$ we can always find points $p_q$ with $\left[ k(p_q): k \right] = q^n$ for $n$ large enough, implying the result.
\end{proof}

\begin{proposition}\label{cover}
Let $\mathscr{M}$ be a quasi-separated algebraic stack. There exists a countable family of algebraic spaces $X_n$ with maps $p_n: X_n \rightarrow \mathscr{M}$ of finite type such that the union of these maps is a smooth-Nisnevich cover. 

If $\mathscr{M}$ has affine stabilizer groups at all of its geometric points, and the base field is infinite, we only need a finite number of maps $p_n: X_n \rightarrow \mathscr{M}$.

 In general if $\mathscr{M}$ has affine stabilizers group at all of its geometric points, we only need a finite number of maps to obtain a smooth $m$-Nisnevich cover.

\end{proposition}
\begin{proof}
The first statement is proven in \cite[6.5]{LMB99}. Note that by dropping the geometrically connectedness requirement on the fibres we can extend the covering family to the whole stack.

Let now $\mathscr{M}$ be an algebraic stack with affine stabilizers groups. In \cite[3.5.9]{Kr99} Kresch proves that under the hypothesis of affine stabilizer groups an algebraic stack admits a stratification by quotient stacks $\left[X /G \right]$, where $X$ is an algebraic space and $G$ a linear algebraic group. We may thus suppose that we are working in these hypothesis. Moreover we may suppose our stack is irreducible.

For a quotient stack we have a standard approximation by an algebraic space. Let $V$ be a representation of $G$ such that there is an open subset $U$ of $V$ on which $G$ acts freely, and $V \setminus U$ has codimension two or more in $V$. Then $\left[ (X\times V)/G \right] \xrightarrow{\pi} \left[X/G \right]$ is a vector bundle, and $\left[ (X\times U )/G \right]$ is an algebraic space. Note that given any open subset $\mathscr{V}$ of $\left[ (X\times V)/G \right]$, the restriction of $\pi$ is a smooth-Nisnevich cover of some open substack $\mathscr{U}$ of $\left[ X/G \right]$. First we reduce to an open subset $\mathscr{U}$ of $\left[ X/G \right]$ such that the fiber of all points of $\mathscr{U}$ is nonempty. We can do that as $\pi$ is universally open. Then the fiber of a point $\op{Spec}(k) \xrightarrow{p} \mathscr{U}$ must be a nonempty open subset of $\mathbb{A}^{n}_{k}$ for some $n$. Then by lemma (\ref{openvec}) we know that $U$ is must be a smooth-Nisnevich neighbourhood of $p$ if $k$ is infinite, and a smooth $m$-Nisnevich neighbourhood of $p$ in general.

Using the fact that the family $\lbrace X_n \rightarrow \mathcal{M} \rbrace_{n \in \mathbb{N}}$ described in \cite[6.5]{LMB99} is functorial we can reduce to the case of a quotient stack $\left[ X / G \right]$. Consider the equivariant approximation map $U/G \rightarrow \left[X/G\right]$ described above. We may restrict to an open subset and suppose that it is surjective and thus respectively a smooth-Nisnevich covering if $k_0$ is infinite or a smooth $m$-Nisnevich covering if not.

Consider the map $\cup_{n} X_n \rightarrow \left[ X / G \right]$. We can take the fiber product with $U/G$ obtaining the following cartesian diagram:

\begin{center}

$\xymatrixcolsep{5pc}
\xymatrix{ \cup_{n} X_n\times_{\left[ X/G \right]} U/G \ar@{->}[r]^{p'} \ar@{->}[d]^{\pi'} & U/G \ar@{->}[d]^{\pi}\\ 
\cup_{n} X_n  \ar@{->}[r]^{p} & \left[ X/G \right]}$

\end{center}

All the arrows in the diagram are smooth-Nisnevich covers (resp. smooth $m$-Nisnevich). Note now that $U/G$ is an algebraic space, and by noetherianity it only takes a finite number of the schemes $X_n\times_{\left[ X/G \right]} U/G$ to cover it. Let $N$ be big enough for this to happen, and consider the new diagram:

\begin{center}

$\xymatrixcolsep{5pc}
\xymatrix{ \cup_{n \leq N} X_n\times_{\left[ X/G \right]} U/G \ar@{->}[r]^{p'} \ar@{->}[d]^{\pi'} & U/G \ar@{->}[d]^{\pi}\\ 
\cup_{n\leq N} X_n  \ar@{->}[r]^{p} & \left[ X/G \right]}$

\end{center}

We know that $\pi \circ p'$ is a smooth Nisnevich cover (resp. smooth $m$-Nisnevich), so $p \circ \pi'$ must be too, and this immediately implies that $p$ is a smooth-Nisnevich cover (resp. smooth $m$-Nisnevich). Then we can use noetherianity to conclude.

\end{proof}

Given an algebraic stack $\mathscr{M}$ we denote $\textit{AlStk}/\mathscr{M}$ the 2-category consisting of representable maps of algebraic stacks $\mathscr{N} \rightarrow \mathscr{M}$ with morphisms given by 2-commutative squares over the identity of $\mathscr{M}$. As we are requiring all maps to be representable, it is equivalent to a 1-category.

We define the (very big) smooth-Nisnevich site $(\textit{AlStk}/\mathscr{M})_{\text{sm-Nis}}$ by allowing all smooth-Nisnevich maps as covers.

We define the (very big) smooth $m$-Nisnevich site $(\textit{AlStk}/\mathscr{M})_{\text{sm m-Nis}}$ by allowing all smooth $m$-Nisnevich maps as covers.

\begin{lemma}\label{sheafNis}
Let $m$ be a number divisible by $p$. Let $K$ be a field and let $\mathcal{U} \rightarrow \op{Spec}(K)$ be an $m$-Nisnevich cover of $\op{Spec}(K)$. Then we have:
\[ \check{H}^0(\mathcal{U},\op{H}^{\bull})=\op{H}^{\bull}(\op{Spec}(K))\]
That is, the functor $\op{H}^{\bull}$ satisfies the sheaf condition with respect to $m$-Nisnevich covers of spectra fo fields.
\end{lemma}
\begin{proof}
Fix a $m$-Nisnevich cover $U\rightarrow \op{Spec}(K)$. We can restrict to a finite cover. It is going to be of the form \[ \op{Spec}(K_1) \sqcup \dots \sqcup \op{Spec}(K_r) \rightarrow \op{Spec}(K)\] where $K_1,\ldots K_r $ are finite separable extensions of $K$ and \[ (\left[K_1:K\right],\ldots, \left[K_r:K\right],m) := (d_1,\ldots, d_r,m) =1\]

Recall that given any  scheme $Y$ étale over the spectrum of a field $K'$ there is a transfer map $t:\op{H}^{\bull}(Y ) \rightarrow \op{H}^{\bull}(\op{Spec}(K'))$ given by taking for every point $\op{Spec}(E) \rightarrow Y$ the norm map $\op{N}^{E}_{K'}$. This is described in \cite[sec.1, 1.11]{Ro96}

Fix $a_1, \ldots , a_r$ such that $a_1 d_1 + \ldots + a_r d_r \equiv 1 (\op{mod} m)$. For any scheme $Y$ étale over $\op{Spec}(K)$ we define a transfer map $T:\op{H}^{\bull}(Y \times_K U) \rightarrow \op{H}^{\bull}(Y)$ by taking for each $K_i$ the usual transfer map $t:\op{H}^{\bull}(Y \times_K \op{Spec}(K_i)) \rightarrow \op{H}^{\bull}(Y)$ point by point and multiplying it by $a_i$. Using the properties of the norm map it is immediate to check that $T$ is a retraction for the pullback $\op{H}^{\bull}(Y) \rightarrow \op{H}^{\bull}(Y\times_K U)$.

For a non negative integer $s$, let $\mathcal{H}^i_s$ be the sheafification in the $m$-Nisnevich topology of the presheaf $X\rightarrow H^{i}_{\textit{ét}}(X,\mu_p^{s})$. It agrees with the component of degree $i$ of $\check{H}^0(\mathcal{U},\op{H}^{\bull})$ when $s=i$. For all $s$ there is a \v{C}ech to cohomology spectral sequence \[E^{ij}_2 = \check{H}^i(\mathcal{U},\mathcal{H}^j_s)\Rightarrow H_{\textit{ét}}^{i+j}(\op{Spec}(K),\mu_p^s)\]
coming from the covering $U \rightarrow \op{Spec}(K)$. We can restate our claim as saying that for all $s$ we have $H^s_{\textit{ét}}(\op{Spec}(K),\mu_p^s)=\check{H}^0(\mathcal{U},\mathcal{H}^s_s)$. By the spectral sequence above to do so it suffices to prove that $\check{H}^j(\mathcal{U},\mathcal{H}^r_s)=0$ for all $j >0, r\geq 0, s \geq 0$.

Let $\mathcal{U}'$ be the pullback to $U$ of $\mathcal{U}$. As $\mathcal{U}'$ admits a section, the \v{C}ech cohomology groups $\check{H}^j(\mathcal{U}',\mathcal{H}^r_s)$ are zero for $j > 0$. There is a natural pullback map of \v{C}ech complexes between the complex of $\mathcal{U}$ and $\mathcal{U}'$, and the transfer map $T$ defines a retraction of this map. Then this implies that the \v{C}ech cohomology groups $\check{H}^j(\mathcal{U},\mathcal{H}^r)$ inject to $\check{H}^j(\mathcal{U}',\mathcal{H}^r)$ so they must be zero too for $j>0$, proving our claim.
\end{proof}

\begin{theorem}\label{sheaf}
 The functor of cohomological invariants $\op{Inv}^{\bull}$ is a sheaf in the smooth-Nisnevich topology. Moreover, let $m$ be a non negative integer divisible by $p$. Then $\op{Inv}^{\bull}$ is also a sheaf in the smooth $m$-Nisnevich topology.
\end{theorem}
\begin{proof}
We begin with the smooth-Nisnevich case. First, notice that as the cohomological invariants of $\op{Spec}(k)$ are equal to its cohomology, if $\alpha$ is a cohomological invariant of $\mathscr{M}$ and $p \in \mathscr{M}(k)$ a point the pullback $p^*(\alpha)$ is the value of $\alpha$ at $p$. 

Now, let $f: \mathscr{M} \rightarrow \mathscr{N}$ be a Nisnevich cover, and $\alpha$ a cohomological invariant of $\mathscr{M}$ satisfying the gluing condition. Let $q$ be a point of $\mathscr{N}$ and $p, p' : \op{Spec}(k) \rightarrow \mathscr{M}$ two different liftings of $q$. By the gluing conditions, \[\alpha(p) = {\op{Pr}_1}^* (\alpha)(p \times_{\mathscr{N}} p') = {\op{Pr}_2}^* (\alpha)(p \times_{\mathscr{N}} p')= \alpha(p').\] We may thus define a candidate invariant $\beta$ by $\beta(q)=\alpha(p)$, where $p$ is any lifting of $q$. 

It is clear that $\beta$ is a natural transformation between the functor of points of $\mathscr{N}$ and $\op{H}^{\bull}(-)$. We need to prove it has property (\ref{def1}).

Let $R$ be a Henselian DVR, $i: \op{Spec}(R) \rightarrow \mathscr{N}$ a morphism. The induced morphism $\op{Pr}_2 : \mathscr{M} \times_{\mathscr{N}} \op{Spec}(R) \rightarrow \op{Spec}(R)$ is a Nisnevich cover of the spectrum of an Henselian ring, so it has a section. This section provides a map $\op{Spec}(R) \rightarrow \mathscr{M}$. By evaluating $\alpha$ at the image of the generic and closed point of $\op{Spec}(R)$, we obtain the desired result.

The general statement follows from the reasoning above and lemma \ref{sheafNis}. For the last part we only need to notice that if $U \rightarrow \op{Spec}(K)$ is an $m$-Nisnevich cover and $R$ is a Henselian $K$-algebra then the induced pullback map \[\op{H}^{\bull}(\op{Spec}(R)) \rightarrow \op{H}^{\bull}(\op{Spec}(R)\times_{K} U)\] is injective.
\end{proof}

We used such a big category to get the strongest statement and also to have a category with the final object $\op{Id}: \mathscr{M} \rightarrow \mathscr{M}$ as a term of comparison. With the next proposition we see that we can reduce our scope to tamer sites.

\begin{definition}\label{sites}
Denote by $\textit{Spc}/\mathscr{M}$ the category of $\mathscr{M}$-algebraic spaces, with morphisms cartesian squares over the identity of $\mathscr{M}$. Denote by $\textit{Sm}/\mathscr{M}$ the full subcategory of $\textit{Spc}/\mathscr{M}$ consisting of algebraic spaces smooth over $k_0$.
On these two categories we consider the Nisnevich sites $(\textit{Spc}/\mathscr{M})_{\text{Nis}}$ and $(\textit{Sm}/\mathscr{M})_{\text{Nis}}$ where the coverings are étale Nisnevich maps, and the smooth-Nisnevich sites $(\textit{Spc}/\mathscr{M})_{\text{sm-Nis}}$ and $(\textit{Spc}/\mathscr{M})_{\text{sm-Nis}}$ where the covers are smooth-Nisnevich maps.

We also define the corresponding $m$-Nisnevich sites $(\textit{Spc}/\mathscr{M})_{\text{m-Nis}}$, $(\textit{Sm}/\mathscr{M})_{\text{m-Nis}}$, $(\textit{Spc}/\mathscr{M})_{\text{sm m-Nis}}$, $(\textit{Spc}/\mathscr{M})_{\text{sm m-Nis}}$
\end{definition}

The site $(\textit{Sm}/\mathscr{M})_{\text{Nis}}$ could be called ``Lisse-Nisnevich" site in analogy with the usual Lisse-étale site on algebraic stacks, and in fact it is a subsite of $\mathscr{M}_{\text{Lis-ét}}$. We have also defined the big sites in analogy with the approach used in \cite[06TI]{St15} as these are the one we are working with in most of the proofs. the next corollary shows that we get the same result regardless which of these sites we choose.

\begin{corollary}\label{eqtopoi}
The $(\textit{AlStk}/\mathscr{M})_{\text{sm-Nis}}$ (resp. $(\textit{AlStk}/\mathscr{M})_{\text{sm m-Nis}}$) site and the sites defined in (\ref{sites}) all induce the same topos.
\end{corollary}
\begin{proof}
This is a consequence of propositions (\ref{cover}) and (\ref{spezz}) and the chains of inclusions

\[(\textit{Sm}/\mathscr{M})_{\text{Nis}}\subseteq (\textit{Sm}/\mathscr{M})_{\text{sm-Nis}} \subseteq (\textit{AlgSt}/\mathscr{M})_{\text{sm-Nis}} \]

\[(\textit{Spc}/\mathscr{M})_{\text{Nis}}\subseteq (\textit{Spc}/\mathscr{M})_{\text{sm-Nis}} \subseteq (\textit{AlgSt}/\mathscr{M})_{\text{sm-Nis}} \]

The same works word by word for the $m$-Nisnevich sites.
\end{proof}

This gives us the tautologic equality $\op{Inv}^{\bull}(\mathscr{M})=H^0((\textit{Sm}/\mathscr{M})_{\text{Nis}},\op{Inv}^{\bull})$. In the next section we will use this equality and the fact that $(\textit{Sm}/\mathscr{M})_{\text{Nis}}$ is a site of smooth algebraic spaces to obtain a satisfactory description of the sheaf $\op{Inv}^{\bull}$.

\section{$\op{Inv}^{\bull}$ as a derived functor}

In this section we give an explicit description of the sheaf of cohomological invariants as the sheafification of the étale cohomology with respect to the smooth-Nisnevich site. This will immediately give us a clear idea on how our invariants should be computed and their properties.

To keep the statements short, we will work on the ordinary Nisnevich and smooth-Nisnevich sites. We can do this without loss of generality as the results for the $m$-Nisnevich sites will be obtained for free from the ordinary case.

\begin{definition}\label{regulinv}
Let $\mathscr{M}$ be an algebraic stack, and let $i: (\textit{Sm}/\mathscr{M})_{\text{Nis}}\rightarrow (\textit{Sm}/\mathscr{M})_{\text{ét}}$ the inclusion of $(\textit{Sm}/\mathscr{M})_{\text{Nis}}$ in the Lisse-étale site of $\mathscr{M}$. It induces a left-exact functor $i_*$ from the Lisse-étale topos of $\mathscr{M}$ to the topos of $(\textit{Sm}/\mathscr{M})_{\text{Nis}}$.

 We will call $\op{RInv}^{\bull} := \oplus_j \op{R}^{j} i_{*} (\mu_p^{\otimes j})$ the sheaf of \emph{regular invariants}.
\end{definition}

We can see the sheaf of regular invariants as the sheafification of the presheaf $U \rightarrow \op{H}^{\bull}(U)$ in any of the sites defined in the previous section.

\begin{remark}
If $R$ is an Henselian ring then $\op{RInv}^{\bull}(\op{Spec}(R))$ is naturally isomorphic to $ \op{H}^{\bull}(\op{Spec}(R))$.
\end{remark}

The map from étale cohomology to cohomological invariants naturally extends to a map of sheaves between regular invariants and cohomological invariants. The previous remark shows that this map can be again interpreted as sending an element $\alpha \in \op{RInv}^{\bull}(\mathscr{M})$ to the cohomological invariant $\tilde{\alpha} \in \op{Inv}^{\bull}(\mathscr{M})$ defined by sending a point $p \in \mathscr{M}(K)$ to $\tilde{\alpha}(p)=p^*(\alpha)$.

\begin{proposition}
The map $\tilde{*} : \op{RInv}^{\bull} \rightarrow \op{Inv}^{\bull}$ is injective.
\end{proposition}
\begin{proof}
Suppose a given regular invariant $\alpha$ is zero as a cohomological invariant. By lemma (\ref{coomhens}), the pullback of a regular invariant to the spectrum of an Henselian local ring is the same as the pullback to its closed point. The fact that $\alpha$ is zero as a cohomological invariant then implies that the pullback of $\alpha$ to the spectrum of any local Henselian ring is zero, as it is zero at its closed point. Then $\alpha$ must be zero, as regular invariants form a sheaf in the Nisnevich topology.
\end{proof}

This shows that we can think of $\op{RInv}^{\bull}$ as a subsheaf of $\op{Inv}^{\bull}$. We want to prove the following:

\begin{theorem}\label{princ}
Let $\mathscr{M}$ be an algebraic stack smooth over $k_0$. Then $\op{RInv}^{\bull}(\mathscr{M})=\op{Inv}(\mathscr{M})$.
\end{theorem}

We will use a few lemmas. First we prove that for a smooth connected space a cohomological invariant is determined by its value at the generic point.

\begin{lemma}\label{Inj}
Let $R$ be a regular Henselian local $k_0$-algebra, with residue field $k$ and quotient field $K$. Let $\alpha$ be a cohomological invariant of $\op{Spec}(R)$. Then if $\alpha(\op{Spec}(K))=0$ we have $\alpha(\op{Spec}(k))=0$.
\end{lemma}
\begin{proof}
We will proceed by induction on the dimension $d$ of $R$. The case $d=0$ is trivial, and the case $d=1$ is proven in \cite[7.7]{GMS03}. Suppose now $d>1$.

Let now $x , \ldots , x_{d-1}$ be a regular sequence for $R$. Set $R_1=\faktor{R}{(x)}$. $R_1$ is Henselian, and by the inductive hypothesis we know that if the value of $\alpha$ at  $\op{Spec}(k(R_1))$ is zero then the value of $\alpha$ at $\op{Spec}(k)$ must be zero too.
 
  Consider $R_2 := R_{(x)}$. The residue field of $R_2$ is $k(R_1)$, and its quotient field is $k(R)$.
 
  Let $R_2^{h}$ be the Henselization of $R_2$, and consider the pullback $\alpha'$ of $\alpha$ through the map $\op{Spec}(R_2^h) \rightarrow \op{Spec}(R)$. We have $\alpha'(\op{Spec}(k(R_1))= \alpha(\op{Spec}(k(R_1)))$, and $\alpha(\op{Spec}(K))=0$ implies the same for the generic point of $\op{Spec}(R_2^h)$. Then we have $\alpha(\op{Spec}(k(R)))=0$, which implies $\alpha(\op{Spec}(k(R_1)))=0$ which in turn implies $\alpha(\op{Spec}(k))=0$.
\end{proof}

\begin{example}
This fails as soon as $X$ is no longer normal. Let $R = \lbrace \phi \in \mathbb{C}\left[\left[ t \right]\right] \mid \phi(0) \in \mathbb{R} \rbrace$. $R$ is an Henselian ring of dimension one, with residue field $\mathbb{R}$ and quotient field $\mathbb{C}\left[\left[ t \right]\right]$, but $\op{H}^{1}(\op{Spec}(\mathbb{R}), \mathbb{F}_{2}) \neq 0$, while $\op{H}^{1}(\op{Spec}(\mathbb{C}\left[\left[ t \right]\right]),\mathbb{F}_2) = 0 $.
\end{example}

\begin{corollary}\label{pgen}
Let $X$ be an irreducible scheme smooth over $k_0$. A cohomological invariant $\alpha$ of $X$ is zero if and only if its value at the generic point of $X$ is zero.
\end{corollary}
\begin{proof}
Let $\alpha$ be a cohomological invariant of $X$ such that its restriction at the generic point $\mu$ is zero. Le $p$ be another point, and let $R$ be the local ring of $p$ in the smooth-Nisnevich topology, $\mu_1$ the its generic point. As $\mu_1$ is obtained by base change from $\mu$, $\alpha(\mu_1)$ must be zero. Then, by the previous lemma, $\alpha(p)$ is zero.
\end{proof}

The same happens for regular invariants.

\begin{lemma}
Let $X$ be a scheme smooth over $k_0$. Let $\mathscr{H}^{\bull}$ be the sheafification of the étale cohomology in the Zariski topology. There is a natural isomorphism of Zariski sheaves $\mathscr{H}^{\bull} \simeq \op{H}_{nr}^{\bull}$ given by restriction to the generic point.
\end{lemma}
\begin{proof}
This is proven by the Gersten resolution \cite[4.2.2]{BO74}.
\end{proof}

By remark $\ref{unramif}$ we know that the value of a cohomological invariant $\alpha$ at the generic point of a smooth space $X$ belongs to the unramified cohomology $H_{nr}(X)$. We only have to put together the previous lemmas to obtain the result for schemes.

\begin{proposition}\label{princschemi}
Let $X$ be a scheme smooth over $k_0$. There is a natural isomorphism $\op{Inv}^{\bull}(X) \simeq \op{H}_{nr}^{\bull}(X)$. In particular, all invariants of $X$ are regular.
\end{proposition}
\begin{proof}
We will prove the proposition for an irreducible smooth scheme. The general statement follows.
Consider these three morphisms: 
\begin{itemize}
\item The map $\tilde{*} : \mathscr{H}(X) \rightarrow \op{Inv}^{\bull}(X)$ given by restricting to points.
\item The map $\op{res}_1 : \mathscr{H}(X) \rightarrow \op{H}_{nr}^{\bull}(X)$ given by restricting to the generic point.
\item The map $\op{res}_2 : \op{Inv}(X) \rightarrow \op{H}_{nr}^{\bull}(X)$ given by evaluating at the generic point.
\end{itemize}
The second map is an isomorphism by the previous lemma, and the third map is injective by corollary \ref{pgen}. As clearly $ \op{res}_2 \circ \, \tilde{*} = \op{res}_1 $, the three maps must all be isomorphisms. 

As $\op{RInv}^{\bull}$ is the Nisnevich sheafification of $U \rightarrow \mathscr{H}(U)$ the result follows.
\end{proof}

\begin{remark}
Proposition \ref{princschemi} implies that given a regular Henselian ring $R$, with closed and generic points $p$ and $P$ respectively, the equation $\alpha(P)=j(\alpha(P))$, as in (\ref{def1}), holds for any cohomological invariant $\alpha$ of $\op{Spec}(R)$. This shows that in the definition of cohomological invariant we could equivalently choose to require the (apparently) stronger property that equation (\ref{def1}) held for all regular Henselian rings, rather than just for DVRs.
\end{remark}

\begin{proof}[Proof of Theorem \ref{princ}]
We can just plug the previous results in the tautological equality $\op{Inv}^{\bull}(\mathscr{M})=H^0((\textit{Sm}/\mathscr{M})_{\text{Nis}},\op{Inv}^{\bull})$ obtaining 

\[\op{Inv}^{\bull}(\mathscr{M})=H^0((\textit{Sm}/\mathscr{M})_{\text{Nis}},\op{Rinv}^{\bull})=H^0((\textit{Sm}/\mathscr{M})_{\text{Nis}},(\op{H}^{\bull})^{\text{Nis}})\]

Where $(\op{H}^{\bull})^{\text{Nis}}$ denotes that we are taking the sheafification in the Nisnevich topology. Then by the standard description of derived functors we get

\[ \op{Inv}^{j}(\mathscr{M}) = \op{R}^{j}i_*(\mu_p^{\otimes j})(\mathscr{M})\]

\end{proof}

\begin{corollary}
The same results hold for the $m$-Nisnevich sites if $p$ divides $m$.
\end{corollary}
\begin{proof}
This is clear as we already know that that cohomological invariants form a sheaf in the finer $m$-Nisnevich topologies.
\end{proof}

We can use the description of the cohomological invariants on schemes to deduce two important properties of cohomological invariants.

\begin{lemma} $ $
\begin{itemize}
\item Let $U \rightarrow X$ be an open immersion of schemes, such that the codimension of the complement of $N$ is at least $2$. Then $\op{H}_{nr}^{\bull}(X)=\op{H}_{nr}^{\bull}(U)$
\item An affine bundle $E \rightarrow X$ induces an isomorphism on unramified cohomology.
\end{itemize}
\end{lemma}
\begin{proof}
The first statement is true by definition and the second is proven in \cite[8.6]{Ro96}.
\end{proof}

\begin{proposition}\label{codim}
 Let $\mathscr{N} \rightarrow \mathscr{M}$ be an open immersion of algebraic stacks, such that the codimension of the complement of $\mathscr{N}$ is at least $2$. Then $\op{Inv}^{\bull}(\mathscr{M})=\op{Inv}^{\bull}(\mathscr{N})$.
\end{proposition}
\begin{proof}

Let $\pi: X \rightarrow \mathscr{M}$ be an element of  smooth-Nisnevich cover of $\mathscr{M}$ by a scheme. As all the elements we will consider belong to $\textit{AlStk}/\mathscr{M}$ we write $A \times B$ for $A \times_{\mathscr{M}} B$. Name $U$ the open subscheme $X \times \mathscr{N}$ of $X$. Consider the commutative diagram:

\begin{center}
 $\xymatrixcolsep{3pc}
\xymatrix{\op{Inv}^{\bull}(\mathscr{M}) \ar@{->}[d]^{i^{*}} \ar@{->}[r]^{\pi^{*}} & \op{Inv}^{\bull}(X) \ar@{->}[d]^{i_{1}^{*}} \ar@/^/[r]^{\op{Pr_1}^{*}} \ar@/_/[r]_{\op{Pr_2}^{*}} & \op{Inv}^{\bull}(X \times X) \ar@{->}[d]^{{i_2}^{*}} \\ 
\op{Inv}^{\bull}(\mathscr{N}) \ar@{->}[r]^{\pi_{1}^{*}} & \op{Inv}^{\bull}(U) \ar@/^/[r]^{\op{Pr_1}^{*}} \ar@/_/[r]_{\op{Pr_2}^{*}} & \op{Inv}^{\bull}(U \times_{\mathscr{N}} U)} 
$
\end{center}

As ${i_1}^{*},{i_2}^{*}$ are isomorphisms (a smooth map fixes codimension), the elements of $\op{Inv}^{\bull}(X)$ satisfying the gluing conditions are the same as those of $\op{Inv}^{\bull}(U)$.

\end{proof}

\begin{proposition}\label{bundle}
Let $\mathscr{M}$ be an algebraic stack smooth over $k_0$. An affine bundle $\rho:\mathscr{V} \rightarrow \mathscr{M}$ induces an isomorphism on cohomological invariants.
\end{proposition}
\begin{proof}
Consider a smooth-Nisnevich cover $f: X \rightarrow \mathscr{M}$. We have a cartesian square

\begin{center}
$\xymatrixcolsep{5pc}
\xymatrix{ \mathscr{V}\times_{\mathscr{M}} X \ar@{->}[d]^{p_1} \ar@{->}[r]^{p_2} & X \ar@{->}[d]^{f}  \\ 
\mathscr{V} \ar@{->}[r]^{\rho} & \mathscr{M} } 
$
\end{center}

The horizontal arrows are affine bundles, and the vertical arrows are smooth-Nisnevich covers. Moreover, we can choose $f$ to trivialize $\mathscr{V}$. The rings of cohomological invariants of $X$ and $\mathscr{V}\times_{\mathscr{M}}X$ are isomorphic, and we can easily see that the gluing conditions hold for an invariant of $\mathscr{V}\times_{\mathscr{M}}X$ if and only if they hold for the corresponding invariant of $X$.
\end{proof} 

By putting together these results we get an alternative proof of Totaro's theorem from \cite[appendix C]{GMS03}:

\begin{theorem}[Totaro]
Let $G$ be an affine algebraic group smooth over $k_0$. Suppose that we have a representation $V$ of $G$ and a closed subset $Z \subset V$ such that the codimension of $Z$ in $V$ is $2$ or more, and the complement $U = V \setminus Z$ is a $G$-torsor. Then the group of cohomological invariants of $G$ is isomorphic to the unramified cohomology of $U/G$. 
\end{theorem}
\begin{proof}
The map $\left[ V/G \right] \rightarrow \op{B}G$ is a vector bundle, so by proposition \ref{bundle} it induces an isomorphism on cohomological invariants by pullback. As $U/G \rightarrow \left[ V/G \right]$ is an open immersion satisfying the requirements of proposition \ref{codim}, it induces an isomorphism on cohomological invariants too.
\end{proof}

Lastly we show a useful spectral sequences that may help in computing cohomological invariants. It uses the equivariant version of Chow groups with coefficients described by Guillot in \cite[sec.2]{Gu08}, in conjunction with the Bloch-Ogus-Rost spectral sequence. It was shown to exist in the case of classical cohomological invariants by Guillot in \cite[sec.5]{Gu08}. Using our machinery it can be easily extended to all quotient stacks. 

Recall that given a scheme $X$, smooth over $k_0$, Rost's Chow ring with coefficients \cite{Ro96} is a bigraded ring $A^{\bull}(X,M)$. The coefficients are taken in a cycle module $M$, which for now will be equal to $\op{H}^{\bull}$. The first grading is given by codimension, and the second grading comes from the cycle module. The component of degree $(i,j)$ is written $A^{i}(X,M^j)$.

 The equivariant Chow ring $A_{G}^{\bull}(X,M)$ extends the notion to the equivariant setting, and can be thought as the Chow ring of the quotient stack $\left[X/G \right]$. In fact it can be proven that it only depends on the stack and not on the specific presentation as a quotient. Finally, the group $A^0(X,\op{H}^{\bull})$ is by definition equal to $\op{H}_{\textit{nr}}^{\bull}(X)$, so by (\ref{princschemi}) for a smooth scheme $X$ we have $$A^{0}(X,\op{H}^{\bull})=\op{Inv}^{\bull}(X).$$

\begin{theorem}[Rost, Bloch and Ogus]\label{robloog}
 Let $G$ be an affine smooth group scheme over $k_0$, acting on a scheme $X$, also smooth over $k_0$. Moreover suppose that $k_0$ contains a primitive root of unit, so that we can identify $\mu_p$ with $\mathbb{Z}/p\mathbb{Z}$. Then there is a first quadrant spectral sequence:
$$E_2^{r,q} = A^r_G(X, \op{H}^{q-r}) \Rightarrow H^{r+q}_{\textit{\'et}}(\left[X/G \right],\mathbb{Z}/p\mathbb{Z})$$
Coming from the Rost-Bloch-Ogus spectral sequence for schemes. Additionally, we have $$A^0_G(X,\op{H}^{\bull})=\op{Inv}^{\bull}(\left[X/G\right]).$$
\end{theorem}
\begin{proof}
The proof in \cite[sec.5]{Gu08} holds almost word by word in our situation. The second statement is a corollary of propositions (\ref{bundle}, \ref{codim}).
\end{proof}

The sequence is particularly useful as $A^r_G(X, \op{H}^{q-r})=0$  when $q-r < 0$, so the groups appearing in it are zero under the main diagonal.

\begin{corollary}
Let $G$ and $X$ be as above. Suppose that $k_0$ contains a primitive $p$-th root of unit. We have $$\op{Inv}^{1}(\left[ X/G \right])=H^{1}_{\textit{\'et}}(\left[X/G \right],\mathbb{Z}/p\mathbb{Z}).$$

 Moreover, $\op{Inv}^{2}(\left[ X/G \right])$ injects into $H^{2}_{\textit{\'et}}(\left[X/G \right],\mathbb{Z}/p\mathbb{Z})$.
\end{corollary}
\section{The invariants of $\mathscr{M}_{1,1}$}

As a first application of the results in this section, we compute the cohomological invariants of the stack $\mathscr{M}_{1,1}$ of elliptic curves. The computation will use Rost's Chow ring with coefficients.

Recall again that for a smooth scheme $X$ we have $\op{Inv}^{\bull}(X)=A^{0}(X,\op{H}^{\bull})$, where $A^{0}(X,\op{H}^{\bull})$ is the $0$-codimensional component of the Chow ring with coefficients, and the coefficients functor is étale cohomology. In the proof we will shorten $A^{\bull}(X,\op{H}^{\bull})$ to $A^{\bull}(X)$.

Also recall that an algebraic group $G$ is called \emph{special} if any $G$-torsor is Zariski-locally trivial. This implies that for any $X$ being acted upon by $G$, the map $X \rightarrow \left[X/G \right]$ is a smooth-Nisnevich cover. Examples of special groups are $GL_n$ and $Sp_n$.

\begin{theorem}\label{ellip}
Suppose the characteristic of $k_0$ is different from two or three. Then the cohomological invariants of $\mathscr{M}_{1,1}$ are trivial if $(p,6)=1$.

 Otherwise, \[\op{Inv}^{\bull}(\mathscr{M}_{1,1})\simeq\faktor{\op{H}^{\bull}(k_0)\left[ t \right]}{(t^2 - \lbrace -1 \rbrace t)}\] where the degree of $t$ is $1$.
 
  The generator $t$ sends an elliptic curve over a field $k$ with Weierstrass form $y^2 = x^3 + ax + b$ to the element $\left[ 4a^3 + 27b^2 \right] \in k^{*}/k^{*p} \simeq \op{H}^{1}(\op{Spec}(k))$.
\end{theorem}
\begin{proof}
Recall that if the characteristic of $k_0$ is different from $2$ and $3$ we have $\mathscr{M}_{1,1} \simeq \left[ X/G_m \right]$, where $X := \mathbb{A}^{2} \setminus \lbrace 4x^3 = 27Y^2 \rbrace$ and the action of $G_m$ is given by $(x,y,t) \rightarrow (xt^4,yt^6)$.

We will first determine the invariants of $X$.  As the multiplicative group is special, $X \rightarrow \mathscr{M}_{1,1}$ is a smooth-Nisnevich cover, so after we compute $\op{Inv}^{\bull}(X)$ all we have to do is check the gluing conditions.

Consider now the closed immersion $G_{m} \rightarrow \mathbb{A}^{2} \setminus (0,0)$ induced by the obvious map $G_m \rightarrow \lbrace 4x^3 = 27Y^2 \rbrace \setminus (0,0)$ given by the normalization $\mathbb{A}^{1} \rightarrow \lbrace 4x^3 = 27Y^2 \rbrace$. We have an exact sequence \cite[sec.5]{Ro96}

\begin{center}
$0 \rightarrow \op{A}^{0}(\mathbb{A}^2 \setminus (0,0)) \rightarrow \op{A}^{0}(X) \xrightarrow{\partial} \op{A}^{0}(G_m) \rightarrow \op{A}^{1}(\mathbb{A}^2 \setminus (0,0)) \rightarrow \op{A}^{1}(X) \xrightarrow{\partial} \op{A}^{1}(G_m)$
\end{center}

To compute $\op{A}^{\bull}(\mathbb{A}^{2} \setminus (0,0))$, we use a second exact sequence:

\begin{center}
$\op{A}^{1}(\mathbb{A}^2) \rightarrow \op{A}^{1}(\mathbb{A}^2 \setminus (0,0)) \rightarrow \op{A}^{0}(\op{Spec}(k)) \rightarrow \op{A}^{2}(\mathbb{A}^2)$
\end{center}

As $\mathbb{A}^{2}$ is a vector bundle over a point, the first and last term are zero, implying that $\op{A}^{1}(\mathbb{A}^2 \setminus (0,0))$ is generated as a $\op{A}^{\bull}(\op{Spec}(k_0))$-module by a single element in degree one. By the results in the previous section, we have  $\op{A}^{0}(\mathbb{A}^2 \setminus (0,0)) = \op{A}^{0}(\mathbb{A}^2) = \op{A}^{*}(\op{Spec}(k_0))$. As $\op{A}^{2}(\mathbb{A}^{2})$ is zero, the continuation to the left of the sequence above implies that also $\op{A}^{2}(\mathbb{A}^2 \setminus (0,0))$ is zero. 

Using the same technique we find that $\op{A}^{0}(G_m , \op{H}^{\bull})$ is a free $\op{H}^{\bull}(\op{Spec}(k_0)$-module generated by the identity in degree zero and an element $t$ in degree one.

We can now go back to the first exact sequence. Using all the data we obtained, we find that $\op{A}^{0}(X)$ is generated as a free $\op{A}^{\bull}(\op{Spec}(k_0))$ module by the identity and an element $\alpha$ in degree one.

Finding out what $\alpha$ is turns out to be easy, as $\op{H}^{1}(X)$ is generated by $\left[ 4x^3 + 27y^2 \right]$, seen as al element of $\mathscr{O}_{X}^{*}/\mathscr{O}_{X}^{*p}$, and this is clearly nonzero as a cohomological invariant.

The last thing we need to do is to check the gluing conditions; let $m: X\times_{\mathscr{M}_{1,1}} X = X \times G_m \rightarrow X$ be the multiplication map, and let $ \pi : X\times G_m \rightarrow X$ be the first projection. Consider the points $q,q':\op{Spec}(k(x,y,t))\rightarrow X$ defined respectively by $x \rightarrow x, y \rightarrow y$ and $x \rightarrow t^4 x , y \rightarrow t^6 y $. The values of $\alpha$ at $q$ and $q'$ are respectively equal to $m^{*}(\alpha)(\mu)$ and $\pi^{*}(\alpha)(\mu)$, where $\mu$ is the generic point of $X \times G_m$. It is necessary and sufficient for $\alpha$ to verify the gluing conditions that $\alpha(q)=\alpha(q')$.

We have $\alpha(q)= \left[ 4x^3 + 27y^2 \right]$, $ \alpha(q')= \left[ t^{12}(4x^3 + 27y^2) \right]$. These two elements of $\op{H}^{1}(\op{Spec}(k(x,y,t))$ are clearly equal if and only if $p$ divides $12$.

Finally, the relation $\alpha^2 = \lbrace -1 \rbrace \alpha$ is due to the fact that when we identify $\op{H}^1(k)$ with $k^*/(k^*)^p$ we have $\lbrace a \rbrace^2 = \lbrace -1 \rbrace \lbrace a \rbrace $ for any $a \in k^*$. One can see this as a consequence of the relations in Milnor's $K$-theory, as the morphism from Milnor's $K$-theory of $k$ to $\op{H}^{\bull}(k)$ is surjective in degree $1$, see \cite[pag.327 and rmk. 1.11]{Ro96}.
\end{proof}

 \section{Generalized cohomological invariants}

Our aim in this section is to construct a smooth-Nisnevich sheaf of cohomological invariants with values in a given cycle module $M$, satisfying the same description we have when $M$ is equal to Galois cohomology with torsion coefficients.

There are two main problems to be fixed here: first, we must find a new continuity condition, and secondly, we need to find a different way to identify the cohomological invariants of a smooth scheme $X$ with $A^0(X,M)$, as we do not have étale cohomology to act as a medium between the two.

In the following we will often restrict to only considering \emph{geometric} discrete valuation rings. This means that the ring $R$ is a $k_0$ algebra and the transcendence degree over $k_0$ of its residue field is equal to that of its quotient field minus one. This is the same as asking that our DVR is the local ring of an irreducible variety at a regular point of codimension one.

\begin{lemma}
Let $R$ be a DVR with valuation $v$, with quotient field $F$ and residue field $k$. Let $M$ be a cycle module. We denote by $K(v)$ the ring $K(F)/(1+m_v)$, where $K(F)$ is Milnor's $K$-theory of $F$, and by $M(v)$ the $K(v)$-module $M(k)\otimes_{K(k)} K(v)$.

 There are maps $i: M(k(v)) \rightarrow M(v)$ and $p: M(k(R)) \rightarrow M(v)$, and a split exact sequence:
 
 \[ 0 \rightarrow M(k(v)) \xrightarrow{i} M(v) \xrightarrow{\partial} M(k(v)) \rightarrow 0 \]
\end{lemma}
\begin{proof}
This is \cite[rmk. 1.6]{Ro96}.
\end{proof}
 
 one easily sees that when $M$ is equal to Galois cohomology and $R$ is Henselian then $M(F)=M(v)$ and the continuity condition for a cohomological invariant is equivalent to asking that $p(F)=i(\alpha(k))$. This motivates the following definition:
 
 \begin{definition}
 Let $\mathscr{X}$ be an algebraic stack, and $M$ a cycle module. A cohomological invariant for $\mathscr{X}$ with coefficients in $M$ is a natural transformation between the functor of points
 
 \[ F_{\mathscr{X}}:(\mbox{field}/k_0) \rightarrow (\mbox{set}) \]
 
 And $M$, satisfying the following property: given a Henselian DVR $R$ as above and a map $\op{Spec}(R) \rightarrow \mathscr{X}$ we have $p(F)=i(\alpha(k))$ in $M(v)$. We denote the functor of cohomological invariants with coefficients in $M$ by $\op{Inv}^{\bull}(-,M)$.
 \end{definition}
 
\begin{remark}\label{sv}
The condition above is equivalent to asking that for any irreducible normal scheme $X \rightarrow \mathscr{X}$ the value of $\alpha$ at $k(X)$ belongs to $A^0(X)$, and that for a Henselian DVR $R$ we have $\alpha(k(v))=s_v(\alpha(k(R)))$, where $s_v$ is the map defined in \cite[sec.1, D4]{Ro96}.
\end{remark}
 
\begin{theorem}\label{sheafgeneral}
Cohomological invariants with coefficients in $M$ form a sheaf in the smooth-Nisnevich topology and smooth $0$-Nisnevich topology.
\end{theorem}
\begin{proof}
We can just repeat word by word the reasoning in (\ref{sheaf}).
\end{proof}

\begin{lemma}
When $X$ is a scheme  smooth over $k_0$ we have an injective map $A^0(X,M) \rightarrow \op{Inv}(X,M)$ assigning to an element $\alpha \in A^0(X,M)$ the invariant $\tilde{\alpha}$ defined by $\tilde{\alpha}(p)=p^*(\alpha)$.
\end{lemma}
\begin{proof}
The only thing to check here is that the continuity condition is satisfied.

First one should note that by blowing up the image of the closed point we can restrict to geometric DVRs. Given a geometric DVR $(R,v)$, we can see it as the local ring of a regular variety at a point of codimension one. Then by \cite[11.2]{Ro96} we know that the specialization map $r\circ J(i)$ induced by the inclusion $i$ of the closed point is the same as the map $s_v$.

 It is immediate to verify that the pullback $A^0(X,M)\xrightarrow{f^*} A^0(\op{Spec}(R),M)$ is compatible with the specialization $A^0(\op{Spec}(R),M)\rightarrow A^0(\op{Spec}(k(v)),M)$, so that we have $r\circ J(i) \circ f^*=(f\circ i)^*$ and we can conclude.
\end{proof}

\begin{theorem}\label{A0}
Let $\mathscr{X}$ be an algebraic stack smooth over $k_0$. Then the sheaf of cohomological invariants of $\mathscr{X}$ with coefficients in a cycle module $M$ is isomorphic to the functor $X \rightarrow A^0(X,M)$ in the smooth-Nisnevich (resp smooth $0$-Nisnevich) site of $\mathscr{X}$.
\end{theorem}
\begin{proof}
We can reason as in theorem (\ref{princ}). The map $A^0(X,M) \rightarrow \op{Inv}(X,M)$ has an evident inverse given by taking the value at the generic point.
\end{proof}

\section{Algebraic spaces and Brauer groups}

For smooth schemes we have a very simple description of cohomolgical invariants as the unramified cohomology ring or equivalently the zero-codimensional Chow group with coefficients in a cycle module $M$. The following proposition extends the description to quasi-separated algebraic spaces:

\begin{proposition-definition}
Let $X$ be a quasi-separated algebraic space, smooth over the base field. We may suppose that $X$ is connected, with generic point $\xi$.

 Given a point $p \in X^{(1)}$ we have a boundary map $M(k(\xi))\xrightarrow{\delta_p} M(k(p))$ given by taking any Nisnevich neighbourhood $U$ of $p$ that is a scheme, pulling back to the generic point of $U$ and then taking the usual boundary map.

We define the \emph{unramified cohomology} of $X$ with coefficients in $M$, denoted $A^0(X,M)$ to be the kernel of the map 
\[ M(k(\xi)) \xrightarrow{\oplus \delta_p} \bigoplus_{p\in X^{(1)}} M(k(p))\]

We have an equality

\[ \op{Inv}^{\bull}(X,M)=A^0(X,M)\]
\end{proposition-definition}
\begin{proof}
First note that the boundary map $\delta_p$ does not depend choice of a specific Nisnevich neighbourhood due to the properties of cycle modules.

Consider now any Nisnevich covering $U \rightarrow X$. As $U$ is a smooth scheme we know that its cohomological invariants are equal to $A^{0}(U,M)=A^0(U,M)$. Suppose we have an element $\alpha \in A^0(X,M)$. Consider for each connected component $W$ of $U$ the element $\alpha_W$ obtained by pulling back $\alpha$. It is immediate to check that the element is unramified and it obviously glues to a cohomological invariant of $X$.

Now suppose an element $\alpha' \in A^0(U,M)$ glues to a cohomological invariant of $X$. As for schemes, one of the connected components of $U$ must be a Zariski neighbourhood of $\xi$, which easily implies that we can assume $\alpha'$ is the pullback of some element $\alpha \in M(k(\xi))$. Now saying that $\alpha'$ belongs to $A^0(U,M)$ is equivalent to saying that $\alpha$ belongs to $A^0(X,M)$.
\end{proof}

\begin{remark}
A more general approach would be to define the Chow groups with coefficients $A_{\bull}(X,M)$ of an algebraic space $X$ using the idea above. This is done in detail in the author's Ph.D. thesis \cite[chapter 2]{Pi15}.
\end{remark}

\begin{proposition}
Fix a cycle module $M$. Cohomological invariants with coefficients in $M$ are a birational invariant for quasi-separated algebraic spaces that are smooth and proper over the base field.
\end{proposition} 
\begin{proof}
Let $X,Y$ be algebraic spaces, smooth and proper over $k_0$, and suppose that we have an isomorphism between open subsets $U \subset X, U' \subset Y$. Then $\op{Inv}^{\bull}(X,M), \op{Inv}^{\bull}(Y,M)$ both inject to $\op{Inv}(U)=\op{Inv}(U',M)$.

We want to show that any $\alpha \in \op{Inv}^{\bull}(Y,M)$ is unramified on all points of codimension $1$ on $X$ when seen as an element of $M(k(Y))=M(k(X))$, implying that it must belong to $A^0(X,M)=\op{Inv}^{\bull}(X,M)$.

Suppose that given any point $x \in X^{(1)}$ we have a Nisnevich neighbourhood $U_x$ of $x$, and a map $U_x\rightarrow Y$ factoring through $X$. Then we can pull back an element $\alpha\in A^0(Y,M)$ to $U_x$ and we get the same element of $\op{H}^{\bull}(k(X)=k(Y))$. Then $\alpha$ must be unramified on $U_x$.

To prove the statement above, it just suffices to use a weaker version of Chow's lemma \cite[088S]{St15}, which says that we have a blow up $Y' \rightarrow Y$, where $Y'$ is a proper scheme. As the blowup is birational we still have a map from an open subset of $X$ to $Y'$. We can then apply the valutative criterion for properness to the rational map $U_x \rightarrow Y'$ to conclude.

Applying the reasoning above to both $X$ and $Y$ we get the equality \[\op{Inv}^{\bull}(X)=\op{Inv}^{\bull}(Y,M)\subset M(k(X)=k(Y)) \qedhere\]
\end{proof}

\begin{corollary}
Let $X$ be a smooth proper algebraic space. Suppose that there exists an element $\alpha \in M(k(X))$ that is ramified over an irreducible divisor $D$. Then there cannot be any contraction $X \rightarrow Y$ where $Y$ is proper smooth that contracts $D$.  
\end{corollary}

We want to study the cohomological Brauer group $\op{Br}'(X)=(H^{2}_{\textit{ét}}(X,G_m))_{\textit{tors}}$. First we prove that it is a Nisnevich sheaf.

\begin{proposition}
Let $X$ be a smooth algebraic space. Then functor $U \rightarrow \op{Br}'(U)$ is a sheaf on the Nisnevich site of $X$.
\end{proposition}
\begin{proof}
Consider the Grothendieck spectral sequence associated to the inclusion of sites $X_{\textit{Nis}} \xrightarrow{i} X_{\textit{ét}}$. It reads 

\[ H^p_{\textit{Nis}}(X,R^q_{i_*}(F)) \Rightarrow H^{p+q}_{\textit{ét}}(X,F)\]

When $F=G_m$ the terms in the second diagonal are $H^{2}_{\textit{Nis}}(X,G_m)$, $H^{1}_{\textit{Nis}}(X,R^1_{i_*}(\mathbb{G_m})$ and $H^{0}_{\textit{Nis}}(X,R^0_{i_*}(G_m))$, thus proving that the first two terms are zero will imply our proposition. First note that the second term is the cohomology of the sheafification of the Picard group in the Nisnevich topology, which is zero.

We still need to prove that $H^{2}_{\textit{Nis}}(X,G_m)=0$. Let $\rho: \op{Spec}(k(X)) \rightarrow X$ be the inclusion of the generic point. The usual resolution 

\[ 0 \rightarrow G_m \rightarrow \rho_* \mathcal{O}^* \rightarrow \bigoplus_{p \in X^{(1)}} p_* \mathbb{Z} \rightarrow 0\]

holds in this situation. Given the inclusion $p: \op{Spec}(K) \rightarrow X$ we can consider the Leray-Cartan spectral sequence \cite[1.22.1]{Ni89}
\[ H^{j}_{\textit{Nis}}(X,R^i_{p_*}(F)) \Rightarrow H^{i+j}_{\textit{Nis}}(\op{Spec}(K),F)\]
Due to the fact that the Nisnevich toplogy on the spectrum of a field is trivial, we know that $H^{n}(\op{Spec}(K),F)=0$ for $n>0$. Consequently we also have $R^i_{p_*}(F)=0$ for $i >0$, and this shows that any sheaf in the form $p_*(F)$ is acyclic in the Nisnevich topology.

Finally, the Nisnevich site of a Noehterian algebraic space is a Noetherian site, which implies that taking direct sums commutes with taking cohomology. This shows that both $\rho_* \mathcal{O}^*$ and $\bigoplus_{p \in X^{(1)}} p_* \mathbb{Z}$ are acyclic, allowing us to conclude.
\end{proof}

In the rest of the section we will assume that $k_0$ is a field of characterstic equal to zero, so that $\mu_n$ is an étale sheaf. Note that we can consider the cycle module $\op{H}^{\bull}(n,-1)=\bigoplus_{i} H^{i}_{\textit{ét}}(-,\mu_{n}^{i-1})$, which has the property that in degree two it is exactly the sheafification of the cohomology of $\mu_n$.

\begin{lemma}
Suppose that $\op{char}(k_0)=0$. Denote by $Br'_n$ the $n$-torsion in the cohomological Brauer group. Then for a quasi-separated algebraic space, smooth over $k_0$ we have \[(\op{Br}'(X))_n = \op{Inv}^2(X,\op{H}^{\bull}(n,-1)).\]
\end{lemma}
\begin{proof}
This is an easy consequence of the étale exact sequence
\[ 0 \rightarrow \mu_n \rightarrow G_m \xrightarrow{n} G_m \rightarrow 0\]
and the fact that the Picard group is locally zero in the Nisnevich topology.
\end{proof}

The last propositions allows to easily use the properties of cohomological invariants to prove properties of the cohomological Brauer group.

\begin{proposition}\label{brauer}
Suppose that $\op{char}(k_0)=0$. Then:
\begin{enumerate}
\item (Purity) If $X$ is a smooth algebraic space, and $Z$ is a closed subspace of codimension $2$ or greater, then $\op{Br}'(X)= \op{Br}'(X\smallsetminus Z)$.
\item (Birational invariance) Let $X,Y$ be algebraic spaces smooth and proper over $k_0$. If they are birational to each other then $\op{Br}'(X)=\op{Br}'(Y)$.
\end{enumerate}
\end{proposition}
\begin{proof}
The properties above hold for $Br'$ if and only if they hold for $Br'_n$ for all $n$. Thus they follow immediately from the properties of $\op{Inv}^{\bull}(-,\op{H}^{\bull}(n,-1))$.
\end{proof}

Note that we cannot hope to have birational invariance in the very next case in term of generality, the case of Deligne-Mumford stacks:

\begin{remark}
Consider the scheme $P^1 \times P^1$ and the Deligne-Mumfor stack $\left[ P^1 \times P^1 / \mu_2 \times \mu_2 \right]$, where the action is given by $(\left[ a: b \right],\left[ c:d \right]) \times (\alpha, \beta) \rightarrow (\left[ \alpha a : b\right],\left[ \beta c:d \right])$. These two are proper, smooth, and birational, as $\left[G_m/\mu_2 \right] = G_m$. 

The Brauer group and cohomological invariants of $P^1 \times P^1$ are trivial, but one can easily see, using equivariant theory, that 

\begin{center}$\op{Inv}^2(\left[ P^1 \times P^1 / \mu_2 \times \mu_2 \right],\op{H}^{\bull}(2,-1))=\op{Br}'(\left[ P^1 \times P^1 / \mu_2 \times \mu_2 \right])_2=\mathbb{Z}/2\mathbb{Z}$.
\end{center}
\end{remark}

\end{document}